\documentclass{article}
\usepackage{amssymb}
\usepackage{amsmath}
\usepackage{amsfonts}

\newtheorem{theorem}{Theorem}

\newtheorem{corollary}[theorem]{Corollary}

\newtheorem{definition}[theorem]{Definition}
\newtheorem{example}[theorem]{Example}

\newtheorem{proposition}[theorem]{Proposition}

\newenvironment{proof}[1][Proof]{\noindent\textbf{#1.} }{\ \rule{0.5em}{0.5em}}
\input{tcilatex}
\begin{document}

\title{On Strongly Convex Sets and Farthest Distance Functions}
\author{Juan Enrique Mart\'{\i}nez-Legaz\thanks{%
I gratefully acknowledge financial support by Grant PID2022-136399NB-C22
from MICINN, Spain, and by ERDF, \textquotedblleft A way to make
Europe\textquotedblright , European Union.} \\
Department d'Economia i d'Hist\`{o}ria Econ\`{o}mica\\
Universitat Aut\`{o}noma de Barcelona\\
Spain\\
juanenrique.martinez.legaz@uab.cat}
\date{}
\maketitle

\begin{abstract}
Given a positive number $R,$ a polarity notion for sets in a real Banach
space is introduced in such a way that the second polar of a set coincides
with the smallest strongly convex set with respect to $R$ that contains it.
In this chapter, strongly convex sets are characterized in terms of their
associated farthest distance functions, and farthest distance functions
associated with strongly convex sets are characterized, too.
\end{abstract}

\noindent \textit{Keywords:\ }Strongly convex sets, farthest distance
functions, polarity.

\bigskip

\noindent \textit{2020 Mathematics Subject Classification: 52A05, 26B25.}

\section{Introduction}

Let $\left( E,\left\Vert \cdot \right\Vert \right) $ be a real normed space.
We say that a set $C\subseteq E$ is \emph{strongly convex }with respect to a
positive real number $R$ if it is an intersection of a (possibly empty)
collection of closed balls with radius $R.$ This definition of strong
convexity can be regarded as a generalization, for closed sets in a normed
space, of the Euclidean notion of strongly convex set with respect to $R$
introduced in \cite[Definition 2.1]{V82}, since, by \cite[Theorem 1, (i) $%
\Leftrightarrow $ (iv)]{V82}, a closed set in a Euclidean space is strongly
convex with respect to $R>0$ (in the sense of \cite[Definition 2.1]{V82}) if
and only if it is an intersection of a collection of closed balls with
radius $R.$ Strong convexity is an important notion in applied mathematics;
see, e.g., the recent paper \cite{NNT23}, which lists several applications
to optimization and optimal control, among other fields, and the survey
paper \cite{GI17}, where some additional applications, in particular to the
geometry of Hilbert spaces and set-valued analysis, are mentioned. A strict
separation property for strongly conves sets by balls has been established
in \cite[Theorerm 4.3]{NNT23}. The more general notion of $M$-strongly
convex set, $M\subseteq E$ being convex and closed and satisfying an
additional suitable property, was introduced in \cite{BP00};\ in the
particular case where $M$ is a ball, this notion reduces to strong convexity.

The strong convexity of sets is related to the homonymous concept for
functions. A real-valued function $f$ on the Euclidean space $\mathbb{R}^{n}$
is called strongly convex if there exists $\alpha >0$ such that%
\begin{equation}
f\left( \left( 1-\lambda \right) x+\lambda y\right) \leq \left( 1-\lambda
\right) f\left( x\right) +\lambda f\left( y\right) -\alpha \lambda \left(
1-\lambda \right) \left\Vert x-y\right\Vert ^{2}  \label{str conv}
\end{equation}%
for all $x,y\in \mathbb{R}^{n}$ and $\lambda \in \left[ 0,1\right] .$ It is
easy to see that this condition is equivalent to the convexity of $f-\alpha
\left\Vert \cdot \right\Vert ^{2}$ (see \cite[Proposition 4.3]{V83}). A
function satisfying (\ref{str conv}) is said to be $\alpha $-convex. In \cite%
{V82}, natural notions of local strong convexity of sets and functions are
defined in such a way that a function $f:\mathbb{R}^{n}\rightarrow \mathbb{R}
$ is locally strongly convex if and only if its epigraph%
\begin{equation*}
Epi~f:=\left\{ \left( x,\beta \right) \in \mathbb{R}^{n}\times \mathbb{R}%
:f\left( x\right) \leq \beta \right\}
\end{equation*}%
is locally strongly convex \cite[Theorem 4]{V82}. Another relation between
the strong convexity of functions and sets was established in \cite[%
Proposition 4.14]{V83}, namely, if $f$ is $\alpha $-convex and%
\begin{equation*}
M:=\sup \left\{ \left\Vert y^{\ast }\right\Vert :y^{\ast }\in \partial
f\left( y\right) ,\text{ }y\in b\left( f^{-1}\left( \left( -\infty ,f\left(
x\right) \right] \right) \right) \right\}
\end{equation*}%
(see (\ref{subd}) below for the definition of $\partial $), with $b$
denoting boundary, then $M<+\infty $ and the sublevel set%
\begin{equation*}
f^{-1}\left( \left( -\infty ,f\left( x\right) \right] \right) =\left\{ y\in 
\mathbb{R}^{n}:f\left( y\right) \leq f\left( x\right) \right\}
\end{equation*}%
is strongly convex with respect to $\frac{M}{2\alpha }.$

The other main notion studied in this work is that of the \emph{farthest
distance function} $F_{C}^{\left\Vert \cdot \right\Vert }:\mathbb{R}%
^{n}\rightarrow \mathbb{R}$ to a nonempty compact convex set $C\subseteq 
\mathbb{R}^{n},$ which is defined by%
\begin{equation*}
F_{C}^{\left\Vert \cdot \right\Vert }\left( x\right) :=\max_{c\in
C}\left\Vert x-c\right\Vert .
\end{equation*}%
Clearly, $F_{C}^{\left\Vert \cdot \right\Vert }$ is convex and Lipschitz
with constant $1,$ since it is the pointwise maximum of a collection of
convex functions that are Lipschitz with constant $1.$ It is related to the
Haussdorff distance, as the farthest distance from a point $x$ to a bounded
set $C$ is nothing but the Hausdorff distance between $\left\{ x\right\} $
and $C.$ Several authors have studied the farthest distance function; see,
e.g., \cite{BG14, BI06, F80, NNT23}. Interestingly, despite being convex,
the farthest distance function to a strongly convex set in a Hilbert space
enjoys a generalized concavity property. To be more specific, \cite{BG14}
shows relations between the strong convexity of a set and a weak concavity
property of its associated farthest distance function. According to \cite[%
Definition 1.6]{BG14}, a function $f$ defined on an open subset $U$ of a
Hilbert space $H$ is called weakly concave with the constant $\gamma $ if it
is continuous on $U$ and $f-\frac{\gamma }{2}\left\Vert \cdot \right\Vert
^{2}$ is concave on each convex subset of $U.$ As proved in \cite[Theorem 3.1%
]{BG14}, if $C,$ a subset of a Hilbert space $H,$ is strongly convex with
respect to $R>0$ and $r>R,$ then $F_{C}^{\left\Vert \cdot \right\Vert }$ is
weakly concave with the constant $\gamma :=\frac{1}{r-R}$ on the set $%
T_{C}\left( r\right) :=\left\{ x\in H:F_{C}^{\left\Vert \cdot \right\Vert
}\left( x\right) >r\right\} $ (this set is called the antineighborhood of
radius $r$ for the set $C;$ clearly, it consists of the centers of the balls
with radius $r$ that do not contain $C).$ Conversely, \cite[Theorem 3.4]%
{BG14} states that if $C\subset H$ is a closed convex bounded set, $r>0,$
and $F_{C}^{\left\Vert \cdot \right\Vert }$ is weakly concave on $%
T_{C}\left( r\right) $ with the constant $\gamma >0,$ then $C$ is strongly
convex with respect to $r-\frac{1}{\gamma }.$

The aim of this chapter is to investigate strongly convex sets and their
associated farthest distance functions. Although convex analysis has been
extensively explored in classical monographs \cite{HL93, R70, Z02}, and
various forms of convexity for sets and multifunctions have been studied
(see, e.g., \cite[Section 2.4]{GHTZ03}), strongly convex sets have received
relatively little attention in the literature, despite their core properties
having been addressed in the excellent papers cited above. This work seeks
to provide new insights and complement the existing results. We introduce a
notion of polarity for strongly convex sets with respect to $R>0,$ in such a
way that the second polar of a set coincides with the smallest strongly
convex set with respect to $R>0$ that contains it. We then study the
properties of the farthest distance function associated with a strongly
convex set in $\mathbb{R}^{n}$ and show how to obtain directly, in an easy
way, the farthest distance function of the polar of a strongly convex set
from the one of the given set. Moreover, we characterize strongly convex
sets in terms of their associated farthest distance functions and, vice
versa, the farthest distance functions associated with strongly convex sets.

The terminology and notation used in this chapter is rather standard. The
closed ball with center $c\in E$ and radius $R>0$ is%
\begin{equation*}
B\left( c,R\right) :=\left\{ x\in E:\left\Vert x-c\right\Vert \leq R\right\}
.
\end{equation*}%
Given a norm $\left\Vert \cdot \right\Vert $ in $\mathbb{R}^{n},$ its dual
norm $\left\Vert \cdot \right\Vert _{\ast }$ is defined by%
\begin{equation*}
\left\Vert x^{\ast }\right\Vert _{\ast }:=\max_{\left\Vert x\right\Vert \leq
1}\left\langle x,x^{\ast }\right\rangle ,
\end{equation*}%
with $\left\langle \cdot ,\cdot \right\rangle $ denoting the Euclidean
scalar product. The unit sphere corresponding to this dual norm is%
\begin{equation*}
S_{\ast }:=\left\{ x^{\ast }\in \mathbb{R}^{n}:\left\Vert x^{\ast
}\right\Vert _{\ast }=1\right\} .
\end{equation*}%
The support function $\sigma _{C}:H\rightarrow \mathbb{R}$ of a set $C$ in a
Hilbert space $\left( H,\left\langle \cdot ,\cdot \right\rangle \right) $ is
defined by 
\begin{equation*}
\sigma _{C}\left( x^{\ast }\right) :=\sup_{x\in C}\left\langle x,x^{\ast
}\right\rangle .
\end{equation*}%
The Fenchel conjugate and the Fenchel subdifferential operator of a lower
semicontinuous proper convex function $f:\mathbb{R}^{n}\rightarrow \mathbb{%
R\cup }\left\{ +\infty \right\} $ are the function $f^{\ast }:\mathbb{R}%
^{n}\rightarrow \mathbb{R\cup }\left\{ +\infty \right\} $ and the set-valued
mapping $\partial f:\mathbb{R}^{n}\rightrightarrows \mathbb{R}^{n},$
respectively, defined by%
\begin{equation*}
f^{\ast }\left( x^{\ast }\right) :=\sup_{x\in \mathbb{R}^{n}}\left\{
\left\langle x,x^{\ast }\right\rangle -f\left( x\right) \right\}
\end{equation*}%
and%
\begin{equation}
\partial f\left( x\right) :=\left\{ x^{\ast }\in \mathbb{R}^{n}:f\left(
y\right) \geq f\left( x\right) +\left\langle y-x,x^{\ast }\right\rangle
\quad \forall y\in \mathbb{R}^{n}\right\} .  \label{subd}
\end{equation}%
The ordinary distance function $d_{C}:\mathbb{R}^{n}\rightarrow \mathbb{R}$
to a non-empty closed set $C\subseteq \mathbb{R}^{n}$ is defined by%
\begin{equation*}
d_{C}\left( x\right) :=\min_{c\in C}\left\Vert x-c\right\Vert .
\end{equation*}%
As usual, $\left\Vert \cdot \right\Vert _{\infty }$ will denote the uniform
norm in the space of bounded functions on $\mathbb{R}^{n}.$

The remainder of this chapter consists of three sections. Section 2 reviews
the fundamental properties of strongly convex sets, Section 3 develops a
polarity notion for strongly convex sets with respect to a given radius, and
Section 4 studies the relationship between strongly convex sets and their
associated farthest distance functions.

\section{Some basic properties of strongly convex sets}

We start by observing that it immediately follows from the definition of
strong convexity that a set is strongly convex with respect to $R>0$ if and
only if it is the intersection of all the closed balls with radius $R$ that
contain it. Notice also that the whole space $E$ is strongly convex with
respect to $R,$ since it is the intersection of the empty collection.

The following proposition is immediate.

\begin{proposition}
\label{int}The intersection of an arbitrary collection of strongly convex
sets with respect to $R>0$ is strongly convex with respect to $R,$ too.
\end{proposition}

Our next proposition addresses the strong convexity of balls.

\begin{proposition}
Let $c\in E$ and $R,R^{\prime }>0.$ The closed ball $B\left( c,R^{\prime
}\right) $ is strongly convex with respect to $R$ if and only if $R\geq
R^{\prime }.$
\end{proposition}

\begin{proof}
The "only if" statement is obvious, since $B\left( c,R^{\prime }\right) $
does not contain any ball with a radius larger than $R^{\prime }.$ To prove
the converse, assume that $R\geq R^{\prime }$ and let $x\in E\setminus
B\left( c,R^{\prime }\right) $. We have $B\left( c,R^{\prime }\right)
\subseteq B\left( c+\left( R-R^{\prime }\right) \frac{c-x}{\left\Vert
c-x\right\Vert },R\right) ,$ because every $y\in B\left( c,R^{\prime
}\right) $ satisfies%
\begin{equation*}
\left\Vert y-\left( c+\left( R-R^{\prime }\right) \frac{c-x}{\left\Vert
c-x\right\Vert }\right) \right\Vert \leq \left\Vert y-c\right\Vert
+R-R^{\prime }\leq R.
\end{equation*}%
On the other hand,%
\begin{eqnarray*}
\left\Vert x-\left( c+\left( R-R^{\prime }\right) \frac{c-x}{\left\Vert
c-x\right\Vert }\right) \right\Vert &=&\left\Vert \frac{\left\Vert
c-x\right\Vert +R-R^{\prime }}{\left\Vert c-x\right\Vert }\left( x-c\right)
\right\Vert \\
&=&\left\Vert c-x\right\Vert +R-R^{\prime }>R,
\end{eqnarray*}%
which shows that $x\notin B\left( c+\left( R-R^{\prime }\right) \frac{c-x}{%
\left\Vert c-x\right\Vert },R\right) .$ Thus, $x$ does not belong to the
intersection of all the closed balls with radius $R$ that contain it, and so
we conclude that $B\left( c,R^{\prime }\right) $ is strongly convex with
respect to $R.$
\end{proof}

\begin{corollary}
A set $C\subseteq E$ is strongly convex with respect to $R>0$ if it is an
intersection of a (possibly empty) collection of closed balls with (non
necessarily equal) radiuses less than or equal to $R.$
\end{corollary}

\begin{corollary}
If $C\subseteq E$ is strongly convex with respect to $R>0$ and $R^{\prime
}>R,$ then it is strongly convex with respect to $R^{\prime },$ too.
\end{corollary}

Adapting \cite[Definition 2]{V82} to our more general setting, we will
consider the following definition of $R$\emph{-strongly convex hull}.

\begin{definition}
The $R$-strongly convex hull of a set $C\subseteq E$ is the intersection of
all the strongly convex sets with respect to $R$ that contain $C$.
\end{definition}

The following result is an immediate consequence of Proposition \ref{int}.

\begin{proposition}
The $R$-strongly convex hull of $C\subseteq E$ is the smallest strongly
convex set with respect to $R$ containing $C.$
\end{proposition}

To conclude this section, we briefly consider the case when $E$ is a Hilbert
space. By \cite[Theorem 2.6]{BP00} and an observation in \cite[p. 953]{BG14}%
, a nonempty closed convex set $C$ in a Hilbert space is strongly convex
with respect to $R>0$ if and only if it is a summand of $B\left( 0,R\right)
, $ that is, there exists a closed convex set $D$ such that $C+D=B\left(
0,R\right) .$ An easy consequence of this result is the following
proposition, which is a particular case of \cite[Proposition 4.1.4]{BP00}.

\begin{proposition}
A nonempty closed bounded convex set $C$ in a Hilbert space is strongly
convex with respect to $R>0$ if and only if $R\left\Vert \cdot \right\Vert
-\sigma _{C}$ is convex and continuous.
\end{proposition}

\section{A polarity notion for strongly convex sets}

For $R>0,$ let $\rho _{R}$ be the symmetric binary relation defined on $E$ by%
\begin{equation*}
x\rho _{R}y\Leftrightarrow \left\Vert x-y\right\Vert \leq R.
\end{equation*}%
We consider the polarity associated with $\rho _{R}$ (see \cite{B67} for the
concept of polarity), namely, the $\rho _{R}$\emph{-polar} of $C\subseteq E$
is%
\begin{equation*}
C^{\rho _{R}}:=\left\{ y\in E:x\rho _{R}y\quad \forall x\in C\right\} .
\end{equation*}%
Clearly, $C^{\rho _{R}}$ is nonempty if and only if $C$ is contained in some
ball with radius $R.$ The equality 
\begin{equation}
C^{\rho _{R}}=\dbigcap\nolimits_{x\in C}B\left( x,R\right)  \label{polar}
\end{equation}%
shows that $C^{\rho _{R}}$ is strongly convex with respect to $R.$

\begin{example}
\label{singl}For $c\in E,$ one has $\left\{ c\right\} ^{\rho _{R}}=B\left(
y,R\right) .$
\end{example}

\begin{example}
\label{balls}Let $c\in E$ and $r>0.$ It is not difficult to prove that%
\begin{equation*}
B\left( c,r\right) ^{\rho _{R}}=\left\{ 
\begin{array}{c}
\emptyset \qquad \qquad \qquad \text{if }R<r, \\ 
\left\{ c\right\} \qquad \qquad \quad \text{if }R=r, \\ 
B\left( c,R-r\right) \quad \text{ if }R\geq r.%
\end{array}%
\right.
\end{equation*}
\end{example}

\bigskip

Following \cite[p. 122, Lemma]{B67}, for every $C,D\subseteq E$ one has%
\begin{equation}
C\subseteq D\Rightarrow D^{\rho _{R}}\subseteq C^{\rho _{R}},
\label{ord rev}
\end{equation}%
\begin{equation}
C\subseteq C^{\rho _{R}\rho _{R}},  \label{incl}
\end{equation}%
\begin{equation}
C^{\rho _{R}\rho _{R}\rho _{R}}=C^{\rho _{R}}.  \label{3rd polar}
\end{equation}%
These properties imply that the mapping assigning to any set $C\subseteq E$
its second $\rho _{R}$-polar $C^{\rho _{R}\rho _{R}}$ is a closure operator,
that is, on top of (\ref{incl}), it satisfies the following properties:%
\begin{equation}
C^{\rho _{R}\rho _{R}\rho _{R}\rho _{R}}=C^{\rho _{R}\rho _{R}},
\label{4th polar}
\end{equation}%
\begin{equation}
C\subseteq D\Rightarrow C^{\rho _{R}\rho _{R}}\subseteq D^{\rho _{R}\rho
_{R}}.  \label{mon 2nd polar}
\end{equation}%
Clearly, properties (\ref{4th polar}) and (\ref{mon 2nd polar}) are
immediate consequences of (\ref{3rd polar})\ and (\ref{ord rev}),
respectively. It follows from (\ref{4th polar}) that the sets that are
closed under this operator, that is, the invariant sets under the mapping $%
C\mapsto C^{\rho _{R}\rho _{R}},$ are those of the type $C^{\rho _{R}\rho
_{R}}$ for some $C\subseteq E.$

\begin{theorem}
\label{2nd polar}For $C\subseteq E,$ one has $C^{\rho _{R}\rho _{R}}=C$ if
and only if $C$ is strongly convex with respect to $R.$
\end{theorem}

\begin{proof}
The "only if" statement follows from (\ref{polar}), by replacing $C$ with $%
C^{\rho _{R}}.$ For the "if"statement, we only need to prove the opposite
inclusion to (\ref{incl}). To this aim, let $x\in E\setminus C.$ Since $C$
is strongly convex with respect to $R,$ there exists $c\in E$ such that $%
C\subseteq B\left( c,R\right) $ and $x\notin B\left( c,R\right) ,$ that is, $%
c\in C^{\rho _{R}}$ and not $x\rho _{R}c.$ These two conditions imply that $%
x\notin C^{\rho _{R}\rho _{R}}.$
\end{proof}

\begin{corollary}
For $C\subseteq E,$ the set $C^{\rho _{R}\rho _{R}}$ is the $R$ strongest
convex hull of $C.$
\end{corollary}

\begin{corollary}
\label{C in terms of pol}A set $C\subseteq E$ is strongly convex with
respect to $R>0$ if and only if%
\begin{equation*}
C=\dbigcap\nolimits_{x\in C^{\rho _{R}}}B\left( x,R\right) .
\end{equation*}
\end{corollary}

\begin{proof}
The "if" statement is evident. To prove the "only if" statement, using
Theorem \ref{2nd polar} followed by (\ref{polar}) applied to $C^{\rho _{R}}$
in place of $C,$ we obtain%
\begin{equation*}
C=C^{\rho _{R}\rho _{R}}=\dbigcap\nolimits_{x\in C^{\rho _{R}}}B\left(
x,R\right) .
\end{equation*}
\end{proof}

\bigskip

Since, as observed above, the $\rho _{R}$-polar of any set is strongly
convex with respect to $R,$ Theorem \ref{2nd polar} also yields the
following corollary.

\begin{corollary}
\label{inv}The mapping $C\mapsto C^{\rho _{R}}$ is an involution of the set
of subsets of $E$ that are strongly convex with respect to $R.$
\end{corollary}

The last result in this section establishes a relation between a strongly
convex set with respect to $R>0$ in a Hilbert space and its $\rho _{R}$%
-polar.

\begin{proposition}
If $H$ is a Hilbert space and $C\subset H$ is nonempty and strongly convex
with respect to $R>0,$ then%
\begin{equation}
C-C^{\rho _{R}}=B\left( 0,R\right) .  \label{C-pol}
\end{equation}
\end{proposition}

\begin{proof}
The inclusion $\subseteq $ is obvious (and holds in every Banach space). To
prove the opposite inclusion, let $x\in B\left( 0,R\right) .$ Since $H$ is a
Hilbert space, $C$ is a summand of $B\left( 0,R\right) ,$ that is, there
exists a clsed convex set $B$ such that%
\begin{equation}
C+B=B\left( 0,R\right) .  \label{summ}
\end{equation}%
Hence, the inclusion $\subseteq $ in (\ref{C-pol}) implies, using the
cancellation law for noenmpty bounded closed convex sets, that $-C^{\rho
_{R}}\subseteq B.$ The latter inclusion actually holds as an equality;
indeed, for every $x\in B$ and $c\in C,$ we have $c+x\in C+B=B\left(
0,R\right) ,$ and therefore $\left\Vert -x-c\right\Vert \leq R,$ which shows
that $-x\in C^{\rho _{R}}.$ This proves the inclusion $B\subseteq -C^{\rho
_{R}}$ and hence the equality. Consequently, (\ref{summ}) reduces to (\ref%
{C-pol}).
\end{proof}

\section{The farthest distance function associated with a strongly convex set%
}

The function $\eta _{f}\colon \mathbb{R}^{n}\rightarrow \mathbb{R}\cup
\{+\infty \}$ defined by 
\begin{equation}
\eta _{f}(x^{\ast }):=%
\begin{cases}
\Vert x^{\ast }\Vert _{\ast }f^{\ast }\left( \frac{x^{\ast }}{\Vert x^{\ast
}\Vert _{\ast }}\right) & \text{if }x^{\ast }\neq 0, \\ 
0 & \text{if }x^{\ast }=0.%
\end{cases}
\label{eta}
\end{equation}%
will appear in the caracterizations of farthest distance functions given in
Theorem \ref{char farthest} and Corollary \ref{char farthest str conv}
below. Clearly, it is the positively homogeneous extension of the
restriction of $f^{\ast }$ to the unit sphere $S_{\ast }$. As observed in 
\cite{M24}, for $x^{\ast }\neq 0$ one has%
\begin{equation*}
\eta _{f}(x^{\ast })=\widetilde{f}^{\ast }(\Vert x^{\ast }\Vert _{\ast
},x^{\ast }),
\end{equation*}%
with $\widetilde{f}^{\ast }\colon \mathbb{R}_{+}\times \mathbb{R}%
^{n}\rightarrow \mathbb{R}\cup \{+\infty \}$ being the perspective function
corresponding to $f^{\ast },$ defined by 
\begin{equation*}
\widetilde{f}(u,x^{\ast }):=%
\begin{cases}
u\widetilde{f}\left( \frac{x^{\ast }}{u}\right) & \text{if }u>0, \\ 
+\infty & \text{otherwise;}%
\end{cases}%
\end{equation*}%
for perspective functions, the interested reader may consult \cite{HL93}.
The following result was stated in \cite[Theorem 4.1]{M24} for a norm in $%
\mathbb{R}^{n}$ less than or equal to the Euclidean norm $\Vert \cdot \Vert
_{2}$ but, as proved next, it holds true for an arbitrary norm.

\begin{theorem}
\label{char farthest}Let $\Vert \cdot \Vert $ be any norm in $\mathbb{R}^{n}$%
. For $f\colon \mathbb{R}^{n}\rightarrow \mathbb{R}$, the following
statements are equivalent:
\end{theorem}

\begin{enumerate}
\item There exists a nonempty compact convex set $C\subseteq \mathbb{R}^{n}$
such that $f=F_{C}^{\left\Vert \cdot \right\Vert }$.

\item The following conditions hold:

\begin{enumerate}
\item[(a)] $\partial f(x)\cap S_{\ast }\neq \emptyset $ for every $x\in 
\mathbb{R}^{n}$,

\item[(b)] $\eta _{f}$ is finite-valued and concave.
\end{enumerate}
\end{enumerate}

\begin{proof}
We first observe that the proof of the implication (2) $\Rightarrow $ (1)
given in \cite[Theorem 4.1]{M24} is valid for an arbitrary norm, as it does
not use the assumption that $\Vert \cdot \Vert $ is less than or equal to $%
\Vert \cdot \Vert _{2}.$

To prove the converse, we will use the well known fact that all norms on $%
\mathbb{R}^{n}$ are equivalent, so that for some real number $k>0$ one has%
\begin{equation*}
k\Vert \cdot \Vert \leq \Vert \cdot \Vert _{2}.
\end{equation*}%
The farthest distance function to $C$ corresponding to $k\Vert \cdot \Vert $
is $kF_{C}^{\left\Vert \cdot \right\Vert },$ and one clearly has $\partial
\left( kF_{C}^{\left\Vert \cdot \right\Vert }\right) =k\partial
F_{C}^{\left\Vert \cdot \right\Vert }$ and $\left( k\left\Vert \cdot
\right\Vert \right) _{\ast }=\frac{1}{k}\left\Vert .\right\Vert _{\ast },$
which implies that the unit sphere corresponding to $\left( k\left\Vert
\cdot \right\Vert \right) _{\ast }$ is $kS_{\ast }.$ Therefore, assuming
(1), by \cite[Theorem 4.1, implication (1) $\Rightarrow $ (2)(a)]{M24}
applied to the norm $k\Vert \cdot \Vert ,$ we have $k\partial
F_{C}^{\left\Vert \cdot \right\Vert }\left( x\right) \cap kS_{\ast }\neq
\emptyset $ for every $x\in \mathbb{R}^{n}.$ Since $k\partial
F_{C}^{\left\Vert \cdot \right\Vert }\left( x\right) \cap kS_{\ast }=k\left(
\partial F_{C}^{\left\Vert \cdot \right\Vert }(x)\cap S_{\ast }\right) ,$ we
obtain (a). The proof of (b) follows easily from a straightforward
computation, as we will see next. Consider the function (\ref{eta})
associated with $kF_{C}^{\left\Vert \cdot \right\Vert }$ corresponding to
the norm $k\Vert \cdot \Vert .$ For every $x^{\ast }\in \mathbb{R}%
^{n}\setminus \left\{ 0\right\} ,$ it satisfies%
\begin{eqnarray*}
\eta _{kF_{C}^{\left\Vert \cdot \right\Vert }}\left( x^{\ast }\right)
&=&\left( k\Vert \cdot \Vert \right) _{\ast }\left( x^{\ast }\right) \left(
kF_{C}^{\left\Vert \cdot \right\Vert }\right) ^{\ast }\left( \frac{x^{\ast }%
}{\left( k\Vert \cdot \Vert \right) _{\ast }\left( x^{\ast }\right) }\right)
\\
&=&\frac{\Vert x^{\ast }\Vert _{\ast }}{k}\left( kF_{C}^{\left\Vert \cdot
\right\Vert }\right) ^{\ast }\left( \frac{kx^{\ast }}{\Vert x^{\ast }\Vert
_{\ast }}\right) =\Vert x^{\ast }\Vert _{\ast }\left( F_{C}^{\left\Vert
\cdot \right\Vert }\right) ^{\ast }\left( \frac{x^{\ast }}{\Vert x^{\ast
}\Vert _{\ast }}\right) \\
&=&\eta _{F_{C}^{\left\Vert \cdot \right\Vert }}\left( x^{\ast }\right) .
\end{eqnarray*}%
Since $\eta _{kF_{C}}\left( 0\right) =0=\eta _{F_{C}}\left( 0\right) ,$ we
have $\eta _{kF_{C}}=\eta _{F_{C}};$ consequently, by \cite[Theorem 4.1,
implication (1) $\Rightarrow $ (2)(b)]{M24}, the function $\eta _{kF_{C}}$
is finite-valued and concave, which shows that the implication (1) $%
\Rightarrow $ (2)(b) holds.
\end{proof}

\begin{corollary}
\label{bij}Let $\Vert \cdot \Vert $ be any norm in $\mathbb{R}^{n}$. The
mapping $C\mapsto F_{C}^{\left\Vert \cdot \right\Vert }$ is a bijection from
the set of nonempty compact convex subsets of $\mathbb{R}^{n}$ onto the set
of functions $f\colon \mathbb{R}^{n}\rightarrow \mathbb{R}$ satisfying (a)
and (b) of Theorem \ref{char farthest}. The inverse mapping is $f\mapsto
\Gamma _{f},$ with $\Gamma _{f}:=\left\{ c\in \mathbb{R}^{n}:\left\Vert
x-c\right\Vert \leq f\left( x\right) \quad \forall x\in \mathbb{R}%
^{n}\right\} .$
\end{corollary}

We next give a relation between the farthest distance function to a strongly
convex set with respect to $R>0$ and the ordinary distance function to its $%
\rho _{R}$-polar set.

\begin{proposition}
\label{ub for farthest}If $C\subseteq \mathbb{R}^{n}$ is strongly convex
with respect to $R>0,$ then%
\begin{equation*}
F_{C}^{\left\Vert \cdot \right\Vert }\leq d_{C^{\rho _{R}}}^{\left\Vert
\cdot \right\Vert }+R.
\end{equation*}
\end{proposition}

\begin{proof}
For any $x\in C,$ denoting by $\overline{x}$ the projection of $x$ onto $%
C^{\rho _{R}},$ we have%
\begin{eqnarray*}
F_{C}^{\left\Vert \cdot \right\Vert }\left( x\right) &=&\max_{y\in
C}\left\Vert x-y\right\Vert \leq \max_{y\in C}\left\{ \left\Vert x-\overline{%
x}\right\Vert +\left\Vert \overline{x}-y\right\Vert \right\} \\
&=&d_{C^{\rho _{R}}}^{\left\Vert \cdot \right\Vert }\left( x\right)
+\max_{y\in C}\left\Vert \overline{x}-y\right\Vert \leq d_{C^{\rho
_{R}}}^{\left\Vert \cdot \right\Vert }\left( x\right) +R.
\end{eqnarray*}
\end{proof}

\bigskip

The following immediate corollary shows that the difference between the
farthest distance function to a strongly convex set with respect to $R>0$
and that to its $\rho _{R}$-polar is uniformly bounded.

\begin{corollary}
If $C\subseteq \mathbb{R}^{n}$ is strongly convex with respect to $R>0,$
then $F_{C^{\rho _{R}}}^{\left\Vert \cdot \right\Vert }-F_{C}^{\left\Vert
\cdot \right\Vert }$ is bounded and%
\begin{equation}
\left\Vert F_{C^{\rho _{R}}}^{\left\Vert \cdot \right\Vert
}-F_{C}^{\left\Vert \cdot \right\Vert }\right\Vert _{\infty }\leq R.
\label{unif bounded}
\end{equation}
\end{corollary}

\begin{proof}
By Proposition \ref{ub for farthest}, we have%
\begin{equation}
F_{C}^{\left\Vert \cdot \right\Vert }\leq d_{C^{\rho _{R}}}^{\left\Vert
\cdot \right\Vert }+R\leq F_{C^{\rho _{R}}}^{\left\Vert \cdot \right\Vert
}+R.  \label{ineq 1}
\end{equation}%
Hence, since $C^{\rho _{R}}$ is strongly convex with respect to $R,$ using
Theorem \ref{2nd polar} we obtain%
\begin{equation}
F_{C^{\rho _{R}}}^{\left\Vert \cdot \right\Vert }\leq F_{C^{\rho _{R}\rho
_{R}}}^{\left\Vert \cdot \right\Vert }+R=F_{C}^{\left\Vert \cdot \right\Vert
}+R.  \label{ineq 2}
\end{equation}%
Combining (\ref{ineq 1}) with (\ref{ineq 2}) yields (\ref{unif bounded}).
\end{proof}

\bigskip

Our next result shows how to obtain $F_{C^{\rho _{R}}}^{\left\Vert \cdot
\right\Vert }$ from $F_{C}^{\left\Vert \cdot \right\Vert }$.

\begin{proposition}
\label{farth polar}If $C\subseteq \mathbb{R}^{n}$ is strongly convex with
respect to $R>0,$ then%
\begin{equation*}
F_{C^{\rho _{R}}}^{\left\Vert \cdot \right\Vert }\left( x\right) =\max
\left\{ \left\Vert x-y\right\Vert :F_{C}^{\left\Vert \cdot \right\Vert
}\left( y\right) \leq R\right\} \quad \forall x\in \mathbb{R}^{n}.
\end{equation*}
\end{proposition}

\begin{proof}
For every $x\in \mathbb{R}^{n},$ we have%
\begin{eqnarray*}
F_{C^{\rho _{R}}}^{\left\Vert \cdot \right\Vert }\left( x\right) &=&\max
\left\{ \left\Vert x-y\right\Vert :\left\Vert y-z\right\Vert \leq R\quad
\forall z\in C\right\} \\
&=&\max \left\{ \left\Vert x-y\right\Vert :F_{C}^{\left\Vert \cdot
\right\Vert }\left( y\right) \leq R\right\} .
\end{eqnarray*}
\end{proof}

\bigskip

The following theorem characterizes strongly convex sets in terms of their
associated farthest distance functions. For several totally different
characterizations, see \cite[Theorems 2.2 and 2.3]{NNT23}.

\begin{theorem}
A nonempty compact convex set $C\subseteq \mathbb{R}^{n}$ is strongly convex
with respect to $R>0$ if and only if for every $x\in \mathbb{R}^{n}$ one has%
\begin{equation}
F_{C}^{\left\Vert \cdot \right\Vert }\left( x\right) =\max \left\{
\left\Vert x-y\right\Vert :\left\Vert y-z\right\Vert \leq R\quad \forall
z\in \left( F_{C}^{\left\Vert \cdot \right\Vert }\right) ^{-1}\left( \left[
0,R\right] \right) \right\} .  \label{funct eq}
\end{equation}
\end{theorem}

\begin{proof}
If $C$ is strongly convex with respect to $R,$ then, using Theorem \ref{2nd
polar} and the obvious equality $C^{\rho _{R}}=F_{C}^{\left\Vert \cdot
\right\Vert }\left[ 0,R\right] $ (which is, essentially, the first equality
in \cite[Proposition 3.2.(a)]{NNT23}), for $x\in C$ we obtain%
\begin{eqnarray*}
F_{C}^{\left\Vert \cdot \right\Vert }\left( x\right) &=&\max \left\{
\left\Vert x-y\right\Vert :y\in C^{\rho _{R}\rho _{R}}\right\} \\
&=&\max \left\{ \left\Vert x-y\right\Vert :\left\Vert y-z\right\Vert \leq
R\quad \forall z\in C^{\rho _{R}}\right\} \\
&=&\max \left\{ \left\Vert x-y\right\Vert :\left\Vert y-z\right\Vert \leq
R\quad \forall z\in \left( F_{C}^{\left\Vert \cdot \right\Vert }\right)
^{-1}\left( \left[ 0,R\right] \right) \right\} .
\end{eqnarray*}%
Conversely, (\ref{funct eq}) means that $F_{C}^{\left\Vert \cdot \right\Vert
}=F_{G}^{\left\Vert \cdot \right\Vert },$ where $G:=\dbigcap\limits_{z\in
\left( F_{C}^{\left\Vert \cdot \right\Vert }\right) ^{-1}\left( \left[ 0,R%
\right] \right) }B\left( \left[ 0,R\right] \right) ,$ which, by Corollary %
\ref{bij}, means in turn that $C=G,$ thus showing that $C$ is strongly
convex with respect to $R.$
\end{proof}

\bigskip

It is interesting to observe that, according to (\ref{funct eq}), the
function $F_{C}^{\left\Vert \cdot \right\Vert }$ is fully determined by its $%
R$-sublevel set $\left( F_{C}^{\left\Vert \cdot \right\Vert }\right)
^{-1}\left( \left[ 0,R\right] \right) .$

Using Theorem \ref{char farthest}, we easily obtain the following
characterization of farthest distance functions to strongly convex sets.

\begin{corollary}
\label{char farthest str conv}Let $\Vert \cdot \Vert $ be any norm in $%
\mathbb{R}^{n}$. For $f\colon \mathbb{R}^{n}\rightarrow \mathbb{R},$ the
following statements are equivalent:
\end{corollary}

\begin{enumerate}
\item There exists a nonempty strongly convex set $C\subseteq \mathbb{R}^{n}$
with respect to $R>0$ such that $f=F_{C}^{\left\Vert \cdot \right\Vert }$.

\item The following conditions hold:

\begin{enumerate}
\item[(a)] $\partial f(x)\cap S_{\ast }\neq \emptyset $ for every $x\in 
\mathbb{R}^{n}$,

\item[(b)] $\eta _{f}$ is finite-valued and concave.

\item[(c)] $f\left( x\right) =\max \left\{ \left\Vert x-y\right\Vert
:\left\Vert y-z\right\Vert \leq R\quad \forall z\in f^{-1}\left[ 0,R\right]
\right\} $ for every $x\in \mathbb{R}^{n}.$
\end{enumerate}
\end{enumerate}

\begin{corollary}
\label{bij 2}Let $\Vert \cdot \Vert $ be any norm in $\mathbb{R}^{n}$. The
mapping $C\mapsto F_{C}^{\left\Vert \cdot \right\Vert }$ is a bijection from
the set of nonempty strongly convex set $C\subseteq \mathbb{R}^{n}$ with
respect to $R>0$ onto the set of functions $f\colon \mathbb{R}%
^{n}\rightarrow \mathbb{R}$ satisfying (a), (b) and (c) of Corollary \ref%
{char farthest}.
\end{corollary}

We next consider the transformation that assigns to the farthest distance
function $F_{C}^{\left\Vert \cdot \right\Vert }$ to a strongly convex set $C$
with respect to $R>0$ the farthest distance function $F_{C^{\rho
_{R}}}^{\left\Vert \cdot \right\Vert }$ to its $\rho _{R}$-polar. In this
regard. combining Corollary \ref{char farthest str conv} with Theorem \ref%
{char farthest}, one obtains the following result:

\begin{theorem}
Let $\Vert \cdot \Vert $ be any norm in $\mathbb{R}^{n}$. The mapping $%
f\mapsto f_{R},$ with $f_{R}:\mathbb{R}^{n}\rightarrow \mathbb{R}$ defined by%
\begin{equation*}
f_{R}\left( y\right) :=\max \left\{ \left\Vert y-x\right\Vert :f\left(
x\right) \leq R\right\} ,
\end{equation*}%
is an involution of the set of functions $f\colon \mathbb{R}^{n}\rightarrow 
\mathbb{R}$ satisfying (a), (b) and (c) of Corollary \ref{char farthest}.
\end{theorem}

\begin{proof}
For every $y\in C,$ by Corollary \ref{bij} we have%
\begin{eqnarray*}
f_{R}\left( y\right) &=&\max \left\{ \left\Vert y-x\right\Vert :f^{-1}\left( 
\left[ 0,R\right] \right) \right\} =\max \left\{ \left\Vert y-x\right\Vert
:\left( F_{\Gamma _{f}}^{\left\Vert \cdot \right\Vert }\right) ^{-1}\left( %
\left[ 0,R\right] \right) \right\} \\
&=&\max \left\{ \left\Vert y-x\right\Vert :\left( \Gamma _{f}\right) ^{\rho
_{R}}\right\} =F_{\left( \Gamma _{f}\right) ^{\rho _{R}}}^{\left\Vert \cdot
\right\Vert }\left( y\right) ,
\end{eqnarray*}%
which shows that $f_{R}$ is the farthest distance function to the $\rho _{R}$%
-polar of the unique strongly convex set $C$ with respect to $R$ such that $%
F_{C}^{\left\Vert \cdot \right\Vert }=f.$ Consequently, $f_{RR}:=\left(
f_{R}\right) _{R}$ is the farthest distance function to the $\rho _{R}$%
-polar of the unique set $Y$ such that $F_{Y}=f_{R}.$ By Proposition \ref%
{farth polar}, we have $Y=C^{\rho _{R}},$ and hence, using Theorem \ref{2nd
polar} we obtain%
\begin{equation*}
f_{RR}=F_{Y^{\rho _{R}}}^{\left\Vert \cdot \right\Vert }=F_{C^{\rho _{R}\rho
_{R}}}^{\left\Vert \cdot \right\Vert }=F_{C}^{\left\Vert \cdot \right\Vert
}=f.
\end{equation*}
\end{proof}

\bigskip

Translating Example \ref{singl}, with $c:=0,$ into the farthest functions
language, one gets the following example, which can also be easily obtained
by direct computation.

\begin{example}
For $R>0,$ one has $\left\Vert \cdot \right\Vert _{R}=\left\Vert
.\right\Vert +R$ and, as a consequence, $\left( \left\Vert .\right\Vert
+R\right) _{R}=\left\Vert .\right\Vert .$
\end{example}

Similarly, the farthest functions version of Example \ref{balls} yields our
final example.

\begin{example}
For $R>r\geq 0,$ one has $\left( \left\Vert .\right\Vert +r\right)
_{R}=\left\Vert .\right\Vert +R-r.$
\end{example}

To conclude, we address a question raised by an anonymous reviewer:\ When is
the farthest function strongly convex? Using the following simple result,
one obtains an immediate answer to this question.

\begin{proposition}
\label{no L sc}No real-valued function on a Euclidean space is both
Lipschitz and strongly convex.
\end{proposition}

\begin{proof}
Assume that $f:\mathbb{R}^{n}\rightarrow \mathbb{R}$ is both Lipschitz with
constant $L$ and $\alpha $-convex, and let $x,y\in \mathbb{R}^{n}.$ Take $%
x^{\ast }\in $ $\partial f\left( x\right) $ and $y^{\ast }\in \partial
f\left( y\right) .$ Since the Lipschitz condition implies that $\left\Vert
x^{\ast }\right\Vert \leq L$ and $\left\Vert y^{\ast }\right\Vert \leq L,$
we have%
\begin{equation}
\left\langle x-y,x^{\ast }-y^{\ast }\right\rangle \leq \left\Vert
x-y\right\Vert \left\Vert x^{\ast }-y^{\ast }\right\Vert \leq \left\Vert
x-y\right\Vert \left( \left\Vert x^{\ast }\right\Vert +\left\Vert y^{\ast
}\right\Vert \right) \leq 2L\left\Vert x-y\right\Vert .  \label{Lipsch}
\end{equation}%
On the other hand, by \cite[Proposition 4.10]{V83}, the subdifferential
operator is strongly monotone, namely, we have $\left\langle x-y,x^{\ast
}-y^{\ast }\right\rangle \geq 2\alpha \left\Vert x-y\right\Vert ^{2}.$
Combining this inequality with (\ref{Lipsch}), we obtain $\left\Vert
x-y\right\Vert \leq \frac{L}{\alpha },$ which means that the diameter of $%
\mathbb{R}^{n}$ is less than or equal to $\frac{L}{\alpha },$ an absurd
conclusion.
\end{proof}

\begin{corollary}
Let $\Vert \cdot \Vert $ be any norm in $\mathbb{R}^{n}.$ For any nonempty
compact convex set $C\subseteq \mathbb{R}^{n}$, the function $%
F_{C}^{\left\Vert \cdot \right\Vert }$ is not strongly convex.
\end{corollary}

\begin{proof}
As observed in the Introduction, $F_{C}^{\left\Vert \cdot \right\Vert }$ is
Lipschitz, so the result follows from Proposition \ref{no L sc} after
observing that the Lipschitz character of a function is independent of the
norm, since all norms on $\mathbb{R}^{n}$ are equivalent.
\end{proof}

\bigskip

In spite of the preceding corollary, we have the following simple result.

\begin{proposition}
Let $\Vert \cdot \Vert $ be any norm in $\mathbb{R}^{n}.$ For any nonempty
compact convex set $C\subseteq \mathbb{R}^{n}$, all the nonempty sublevel
sets of $F_{C}^{\left\Vert \cdot \right\Vert }$ are strongly convex, namely,
for every $R>0$ the set $\left( F_{C}^{\left\Vert \cdot \right\Vert }\right)
^{-1}\left( \left[ 0,R\right] \right) $ is strongly convex with respect to $%
R.$
\end{proposition}

\begin{proof}
If $\left( F_{C}^{\left\Vert \cdot \right\Vert }\right) ^{-1}\left( \left[
0,R\right] \right) $ is nonempty, the equality $\left( F_{C}^{\left\Vert
\cdot \right\Vert }\right) ^{-1}\left( \left[ 0,R\right] \right) =C^{\rho
_{R}}$ shows that it is strongly convex with respect to $R.$
\end{proof}

\bigskip

\noindent \textbf{Acknowledgments. }I have been partially supported by Grant
PID2022-136399NB-C22 from MICINN, Spain, and ERDF, \textquotedblright A way
to make Europe", European Union. I am very grateful to an anonymous referee
for his/her thorough review, insightful comments and interesting suggestions.


\begin{thebibliography}{99}
\bibitem{BG14} Balashov, M. V., Golubev, M. O.: Weak concavity of the
antidistance function. J. Convex Analysis \textbf{21}, 951-964 (2014)

\bibitem{BI06} Balashov, M. V., Ivanov, G. E.: On farthest points of sets.
Math. Notes \textbf{80}, 159-166 (2006); translation from Mat. Zametki 
\textbf{80}, 63-170 (2006)

\bibitem{BP00} Balashov, M. V., Polovinkin, E. S.: M-strongly convex subsets
and their generating sets. Sb. Math. \textbf{191}, 25-60 (2000); translation
from Mat. Sb. \textbf{191}, 27-64 (2000)

\bibitem{B67} Birkhoff, G.:\ Lattice theory. American Mathematical Society,
Providence (1967)

\bibitem{F80} Fitzpatrick, S.: Metric projections and the differentiability
of distance functions. Bull. Aust. Math. Soc. \textbf{22}, 291-312 (1980)

\bibitem{GI17} Goncharov, V. V., Ivanov, G. E.: Strong and weak convexity of
closed sets in a Hilbert space. In: Daras, M. et al. (eds.) Operations
Research, Engineering, and Cyber Security. Trends in Applied Mathematics and
Technology, pp. 259-297. Springer, Cham, 259-297 (2017)

\bibitem{GHTZ03} G\"{o}pfert, A., Riahi, H., Tammer, C., Z\u{a}linescu, C.:
Variational Methods in Partially Ordered Spaces. Springer, New York (2003)

\bibitem{HL93} Hiriart-Urruty, J. B., Lemar\'{e}chal, C.: Convex Analysis
and Minimization Algorithms. I: Fundamentals. Springer, Berlin (1993)

\bibitem{M24} Mart\'{\i}nez-Legaz, J. E.: Dual characterizations of three
distance functions. J. Convex Anal. \textbf{31}, 1203-1212 (2024)

\bibitem{NNT23} Nacry, F., Nguyen, V. A. T., Thibault, L.: Farthest distance
function to strongly convex sets, J. Convex Analysis \textbf{30}, 1217--1240
(2023)

\bibitem{R70} Rockafellar, R. T.: Convex Analysis. Princeton University
Press, Princeton (1970)

\bibitem{V82} Vial, J.-Ph.: Strong convexity of sets and functions. J. Math.
Econ. \textbf{9}, 187-205 (1982)

\bibitem{V83} Vial, J.-Ph.: Strong and weak convexity of sets and functions.
Math. Oper. Res. \textbf{8}, 231-259 (1983)

\bibitem{Z02} Z\u{a}linescu, C.: Convex Analysis in General Vector Spaces.
World Scientific, Singapore (2002I
\end{thebibliography}
\end{document}